\documentclass{amsart}

\usepackage{amssymb}
\usepackage{amsthm}
\usepackage{color}

\theoremstyle{plain}
\newtheorem{theorem}{Theorem}[section]

\newtheorem{lemma}[theorem]{Lemma}

\newtheorem*{theorem*}{Theorem}
\newtheorem*{conjecture*}{Conjecture}

\newcommand{\N}{\mathbb{N}}
\newcommand{\Z}{\mathbb{Z}}                                               
\newcommand{\Q}{\mathbb{Q}}                                               
\newcommand{\R}{\mathbb{R}}                                               
\newcommand{\C}{\mathbb{C}} 
\newcommand{\A}{\mathbb{A}} 
\newcommand{\B}{\mathbb{B}} 
\newcommand{\m}{\mathfrak{m}} 
\newcommand{\n}{\mathfrak{n}} 
\newcommand{\aid}{\mathfrak{a}} 
\newcommand{\pid}{\mathfrak{p}}

\newcommand{\g}{\mathfrak{g}}
\newcommand{\sid}{\mathfrak{s}}
\newcommand{\tid}{\mathfrak{t}}

\newcommand{\App}{\mathcal{A}(\PSI)}

\newcommand{\OK}{\mathcal{O}}

\newcommand{\defiso}[5]{$$\begin{array}{rrcl} #1: & #2 & \longrightarrow & #3 \\ & #4 & \longmapsto & #5 \end{array}$$}
\newcommand{\x}{\mathbf{x}}

\newcommand{\PSI}{\underline{\psi}}

\newcommand{\z}{\mathbf{z}}

\DeclareMathOperator{\dnm}{dnm}
\DeclareMathOperator{\vol}{vol}

\newcommand{\aand}{\hbox{\quad and \quad}}

\DeclareMathOperator{\Nm}{N}

\DeclareMathOperator{\Ball}{Bx}

\newcommand{\rset}[2]{\left\{ #1 \ \left| \ #2 \right. \right\}}

\begin{document}

\title{The Duffin--Schaeffer theorem in number fields}
\author{Matthew Palmer}
\thanks{Research supported by an EPSRC Doctoral Training Grant (EP/K502996/1) and by the Heilbronn Institute for Mathematical Research. \\ \indent The results in this paper formed part of the author's doctoral thesis at the University of Bristol (\cite{Palmer2016}). }
\address{Department of Mathematics, Box 480, Uppsala University, SE-75106 Uppsala, Sweden}
\email{matthew.palmer@math.uu.se}

\begin{abstract}
The Duffin--Schaeffer theorem is a well-known result from metric number theory, which generalises Khinchin's theorem from monotonic functions to a wider class of approximating functions.

In recent years, there has been some interest in proving versions of classical theorems from Diophantine approximation in various generalised settings. In the case of number fields, there has been a version of Khinchin's theorem proven which holds for all number fields (\cite{Cantor1965}), and a version of the Duffin--Schaeffer theorem proven only in imaginary quadratic fields (\cite{NakadaWagner1991}).

In this paper, we prove a version of the Duffin--Schaeffer theorem for all number fields.
\end{abstract}

\maketitle

In a 1941 paper (\cite{DuffinSchaeffer1941}), R. J. Duffin and A. C. Schaeffer stated a conjecture, now famous in Diophantine approximation and metric number theory as the \emph{Duffin--Schaeffer conjecture}:

\begin{conjecture*}[Duffin \& Schaeffer, 1941] Suppose that a function \( \psi : \N \to \R_{\geq 0} \) satisfies the condition
\begin{equation}\label{eq:divC} \sum_{n \in \N} \frac{\psi(n) \varphi(n)}{n} = \infty, \end{equation}
where \( \varphi \) is the Euler totient function. Then the set \( A(\psi) \) defined by
\[ A(\psi) = \rset{x \in [0,1]}{\left| x - \frac{a}{n} \right| \leq \frac{\psi(n)}{n} \ \hbox{for infinitely many reduced } \frac{a}{n} \in \Q} \]
is of Lebesgue measure \(1\). \end{conjecture*}

In the same paper, Duffin and Schaeffer proved the following partial result towards this conjecture:

\begin{theorem*}[Theorem I, \cite{DuffinSchaeffer1941}] Suppose that \eqref{eq:divC} holds, and that \( \psi \) also satisfies
\begin{equation} \label{eq:classicallimsup} \limsup_{N \to \infty} \frac{\sum_{n=1}^N \frac{\psi(n) \varphi(n)}{n}}{\sum_{n=1}^N \psi(n)} > 0.\end{equation}

Then the set \(A(\psi)\) has Lebesgue measure \(1\).\end{theorem*}

This result generalises a 1924 result of Khinchin (\cite{Khintchine1924}), which did not require the fractions \(\frac{a}{n}\) to be monotonic, and has in place of \eqref{eq:classicallimsup} the condition that \(n\psi(n)\) is monotonically decreasing.

Since the Duffin--Schaeffer conjecture was stated, a lot of effort has gone into trying to prove it, or at least prove better partial results towards it. (See, for example, \cite{Erdos1970}, \cite{Vaaler1978}, \cite{Harman1990}, \cite{ED1}, \cite{ED2}, \cite{Aistleitner2014}, \cite{AistleitnerEtAl}.) In fact, during revisions to this paper, a proof of the full conjecture was announced by Dimitris Koukoulopoulos and James Maynard (see \cite{KMDSProof}).

However, work has also gone into trying to identify and prove analogues and natural generalisations of the main theorems (including those of Khinchin and of Duffin and Schaeffer) in different setups. One such very natural generalisation is to replace the rationals by a generic number field $K$, and to approximate elements of its various completions by elements of $K$. In 1965, Cantor (\cite{Cantor1965}) proved a version of Khinchin's theorem in this setup for general number fields, and in 1991 Nakada and Wagner (\cite{NakadaWagner1991}) proved a version of the Duffin--Schaeffer theorem for imaginary quadratic fields.

In this paper, we prove a version of the Duffin--Schaeffer theorem for general number fields.

In \S\ref{sec:setup}, we will lay out the setup we will be working in, give the results of Cantor and of Nakada and Wagner, and state our main result (namely Theorem \ref{thm:duffinschaefferK}).

In \S\ref{sec:zeroone}, we will prove a version of Gallagher's classical zero-one law (see Theorem 1 in \cite{Gallagher1961}) in our number field setup (Theorem \ref{thm:zerooneK}), and in \S\ref{sec:overlap}, we will prove some useful overlap estimates (Lemma \ref{lem:overlapK}). Finally, in \S\ref{sec:mainproof}, we will use Theorem \ref{thm:zerooneK} and Lemma \ref{lem:overlapK} to prove Theorem \ref{thm:duffinschaefferK}.

\emph{Notation.} In this paper, the set \(\N\) of natural numbers does not include \(0\).

\tableofcontents
\section{Setup and main result}
\label{sec:setup}
In this section, we describe some of the existing results in the field, before going on to describe the setup we will be working in and state our main theorem.

\subsection{Existing results}
The first work towards a version of the Duffin--Schaeffer theorem in number fields was done by Cantor, who proved a very general version of Khinchin's theorem in number fields (see Theorem 5.12 in \cite{Cantor1965}). Later, in 1991, Nakada and Wagner proved the following version of the Duffin--Schaeffer theorem for imaginary quadratic fields:

\begin{theorem}[Theorem 2, \cite{NakadaWagner1991}] Let $K = \Q(\sqrt{D})$, where $D$ is a square-free negative integer, and let $\psi$ be a non-negative function defined on the ring of integers $\OK_K$ of \(K\) which satisfies $\psi(\gamma) = \psi(u \cdot \gamma)$ for all units $u \in \OK_K^{\times}$. Denote by $\Phi(\gamma)$ the Euler function of $K$, i.e. the number of reduced residue classes mod $\gamma$.

Now suppose that
$$\sum_{\gamma \in \OK_K} \psi(\gamma)^2 = \infty$$
and that for some constant $C > 0$, there exist infinitely many $R \in \N$ such that
$$\sum_{\substack{|\gamma| < R \\ \gamma \in \OK_K}} \psi(\gamma)^2 < C \sum_{\substack{|\gamma| < R \\ \gamma \in \OK_K}} \frac{\psi(\gamma)^2 \Phi(\gamma)}{|\gamma|^2}.$$

Then the inequality
$$\left|z - \frac{\alpha}{\gamma}\right| < \frac{\psi(\gamma)}{|\gamma|}, \quad (\alpha,\gamma) = 1, \quad \alpha,\gamma \in \OK_K$$
has infinitely many solutions for almost all $z \in \C$. \end{theorem}

However, we have some issues with this result. Namely, it does not allow for all elements of $K$ to be used as approximants. While any element of $\Q$ can be written as $\frac{a}{n}$ for some $a,n \in \Z$ with $(a,n) = 1$, this is a fact that comes from uniqueness of factorisation, and hence the same cannot be said for a general element of an imaginary quadratic field, where we can have class number greater than $1$. One famous example of non-unique factorisation is in $K = \Q(\sqrt{-5})$, where we have
$$2 \cdot 3 = 6 = (1 + \sqrt{-5})(1 - \sqrt{-5}),$$
and hence the element $\frac{1 + \sqrt{-5}}{2}$ has no unique reduced form as a quotient of elements. This suggests that the right way to state these sorts of results is by considering not elements, but \emph{ideals}.
\subsection{Diophantine approximation in general number fields}
Let $K$ be a number field of degree $n$. Let $\OK_K$ denote its ring of integers, and let $I_K$ denote the semigroup of ideals of $\OK_K$. We define the Euler $\Phi$-function on \(I_K\) by
\defiso{\Phi}{I_K}{\N}{\n}{\#(\OK_K / \n \OK_K)^{\times},}
where \(R^{\times}\) denotes the group of units in a ring \(R\), and \(\# A \) denotes the cardinality of a finite set \(A\).

Suppose that $K$ has $s$ real embeddings and $t$ pairs of complex embeddings, and denote them by $\sigma_1,\ldots,\sigma_s$ and $\tau_1,\ldots,\tau_t$ respectively. We denote the set of all embeddings of $K$ by $\Sigma$, and denote a generic embedding by $\rho$.

We also define $|\cdot|_{\R}$ to be the standard real absolute value, and $|\cdot|_{\C}$ to be the \emph{square} of the standard complex absolute value. Then we define $|\cdot|_{\rho}$ to be either $|\cdot|_{\R}$ if $\rho$ is real or $|\cdot|_{\C}$ if $\rho$ is complex. (If we take the absolute value of something explicitly involving $\rho$, for example $|\rho(\gamma)|$ or $|x - \rho(\gamma)|$, we assume that the absolute value is with respect to $\rho$, and hence omit the subscript.)

For any element $\gamma \in K$, we define the norm $\Nm(\gamma)$ of $\gamma$ by
$$\Nm(\gamma) = \prod_{\hbox{\scriptsize $\rho$ real}} \rho(\gamma) \prod_{\hbox{\scriptsize $\rho$ complex}} \rho(\gamma) \overline{\rho(\gamma)}.$$

We identify each element of $K$ with an element of $\R^s \times \C^t$ by embedding it into each of its completions. That is to say, we define a map $\iota : K \to \R^s \times \C^t$ by
$$\iota(\alpha) = (\sigma_1(\alpha),\ldots,\sigma_s(\alpha),\tau_1(\alpha),\ldots,\tau_t(\alpha)).$$

The image $\iota(\OK_K)$ of $\OK_K$ under this map forms a lattice in $\R^s \times \C^t$. We fix a fundamental domain of this lattice, and denote it by $D_K$. We have a measure \(\lambda\) on \(D_K\) induced by the Lebesgue measure on \( \R^s \times \C^t \).

As a result of this diagonal embedding of \(K\) into \( \R^s \times \C^t \), we can index the components of an element \( \z \in \R^s \times \C^t \) by the embeddings of \(K\). That is, we can write
\[ \z = (z_1,\ldots,z_s,z_{s+1},\ldots,z_{s+t}) = (z_{\sigma_1},\ldots,z_{\sigma_s},z_{\tau_1},\ldots,z_{\tau_t}).\]

Then for any embedding \( \rho \), we can refer to the \(\rho\)-coordinate \(z_{\rho}\) of an element \( \z \in \R^s \times \C^t \).

By Dirichlet's unit theorem, the group of units of \(\OK_K\) has rank \( s+t-1\). That is to say, there exists a set of multiplicatively independent elements \(\{u_1,\ldots,u_{s+t-1}\} \subset \OK_K^{\times} \) such that any element \(u \in \OK_K^{\times}\) can be written as 
\[ u = \zeta u_1^{n_1} \cdot \cdots \cdot u_{s+t-1}^{n_{s+t-1}}, \]
where \( \zeta \) is some root of unity in \(\OK_K^{\times}\). We call such a set \(\{u_i\}\) a \emph{system of fundamental units} of \(K\).


For each embedding $\rho \in \Sigma$, we choose a function $\psi_{\rho} : I_K \to \R_{\geq 0}$. We combine these into one function $\PSI$ by defining
\defiso{\PSI}{I_K}{\R_{\geq 0}^{s+t}}{\n}{\bigoplus_{\rho \in \Sigma} \psi_{\rho}(\n).}

We also define a function $\Psi : I_K \to \R_{\geq 0}$ by
$$\Psi(\n) = \left( \prod_{\substack{\sigma \in \Sigma \\ \hbox{\scriptsize $\sigma$ real}}} \psi_{\sigma}(\n) \right) \cdot \left( \prod_{\substack{\tau \in \Sigma \\ \hbox{\scriptsize $\tau$ complex}}} \psi_{\tau}(\n)^2 \right).$$

For any element $\gamma \in K$, we have a unique way of writing $(\gamma) = \frac{\aid}{\n}$ with $\aid,\n \in I_K$ and $(\aid,\n) = 1$. Then we write $\dnm \gamma = \n$. 

For $\x \in \R^s \times \C^t$, we say $\gamma \in K$ is a \emph{$\PSI$-good approximation} to $\x$ if we have
$$\left| x_{\rho} - \rho(\gamma) \right| \leq \psi_{\rho}(\dnm(\gamma))$$
for each $\rho \in \Sigma$. We then define a set $\App$ by 
$$\App = \left\{ \x \in D_K \left| \begin{array}{c} \hbox{there exist infinitely many $\gamma \in K$ such} \\ \hbox{that $\gamma$ is a $\PSI$-good approximation to $\x$} \end{array} \right. \right\}.$$

Then our version of the Duffin--Schaeffer theorem for number fields is as follows:

\begin{theorem}\label{thm:duffinschaefferK} If we have
\begin{equation}\label{eq:divergence}\sum_{\n \in I_K} \Phi(\n) \Psi(\n) = \infty\end{equation}
and
\begin{equation}\label{eq:limsup}\limsup_{R \to \infty} \frac{\sum_{\substack{\n \in I_K \\ \Nm(\n) \leq R}} \Phi(\n) \Psi(\n)}{\sum_{\substack{\n \in I_K \\ \Nm(\n) \leq R}} \Nm(\n) \Psi(\n)} > 0,\end{equation}
and $\PSI$ satisfies the boundedness condition
\begin{equation}\label{eq:boundedness}\psi_{\rho}(\n) \leq \frac{1}{2 \Nm(\n)^{\frac{1}{s+t}}}\end{equation}
for all \( \rho \in \Sigma\), then $\App$ has measure $\lambda(D_K)$. \end{theorem}

\textbf{Note.} The boundedness condition \eqref{eq:boundedness} is the equivalent of the (implicit) assumption in Duffin and Schaeffer's original paper that \( \psi(n) \leq \frac{1}{2} \). This assumption was removed in a paper by Pollington and Vaughan (see \cite{PollingtonVaughan}); however, the methods there do not seem to generalise easily to the case of number fields, and hence we state our result with the boundedness condition.

In the next section, we will state and prove a zero-one law for sets of the form $\App$, which will be instrumental in proving Theorem \ref{thm:duffinschaefferK}.

\section{A zero-one law}
\label{sec:zeroone}

The statement we intend to prove is the following:

\begin{theorem}\label{thm:zerooneK} Suppose that $\Psi(\n) \to 0$ as $\Nm(\n) \to \infty$. Then the set $\App$ has measure $0$ or $\lambda(D_K)$. \end{theorem}

Before we can prove this result, we will need a few lemmas.

\begin{lemma}\label{lem:setsinboxes} Let $\{B_n\}_{n \in \N}$ be a sequence of boxes in $\R^s \times \C^t$ such that $\lambda(B_n) \to 0$ as $n \to \infty$, and let $U_n$ be a sequence of measurable sets such that, for some positive $\varepsilon < 1$, we have
$$U_n \subset B_n \aand \lambda(U_n) \geq \varepsilon \lambda(B_n)$$
for each $n \in \N$.

Then we have
\[
\lambda \left( \limsup_{n \in \N} U_n \right) = \lambda(\limsup_{n \in \N} B_n).\]\end{lemma}

\begin{lemma}\label{lem:normbound} For any number field $K$ and constant $C > 0$, there exists a bound $H_K(C)$ such that for all $\gamma \in \OK_K$ with $\Nm(\gamma) > H_K(C)$, there exists some $u \in \OK_K^{\times}$ with
$$|\rho(u \gamma)| > C \hbox{ for all $\rho \in \Sigma$.}$$ \end{lemma}

\emph{Note.} In the rational case and the imaginary quadratic case covered by Nakada and Wagner, this lemma is trivially true, since the number of units is always finite.

Lemma~\ref{lem:setsinboxes} is an analogue of Lemma 2 in \cite{Gallagher1961}, and as the proof follows in exactly the same way, we will not give it here; Lemma~\ref{lem:normbound} follows directly from Lemma 1 in Chapter V of \cite{LangBook}.
%
%
%
Now we apply these two results to prove a final lemma.

\begin{lemma}\label{lem:metricallytransitive} Let $K$ be a number field. Let $\mathcal{F}_K = \{u_1,\ldots,u_r\}$ denote a system of fundamental units of $K$ (where \(r = s+t-1\)), and define $\Omega_K$ to be the constant
$$\Omega_K = \max_{\rho \in \Sigma} \max_{u \in \mathcal{F}_K} |\rho(u)|.$$

For any elements $\alpha,\beta \in \OK_K$ with \(\alpha \neq 0\), define a map $T_{\alpha,\beta} : D_K \to D_K$ by
\[
T_{\alpha,\beta} : x_{\rho} \mapsto \rho(\alpha) x_{\rho} + \rho\left(\tfrac{\beta}{\alpha}\right) \mod \iota(\OK_K).\]

Then if a set $A \subseteq D_K$ satisfies 
\begin{equation} \label{eq:Tinvariance} T_{\pi,\kappa}(A) \subseteq A \aand T_{u,0}(A) \subseteq A \hbox{ for all } u \in \mathcal{F}_K\end{equation}
for some $\pi, \kappa \in \OK_K$ with $|\rho(\pi)| > \Omega_K$ for all $\rho \in \Sigma$, then the set $A$ has measure $0$ or $\lambda(D_K)$. \end{lemma}

\textbf{Remark.} The reader may wonder why we do not prove a similar result simply requiring that our set satisfies \( T_{\pi,\kappa}(A) \subseteq A \) for some \(\pi, \kappa \in \OK_K \) with \( |\rho(\pi)| > 1 \) for all \(\rho \in \Sigma\), without requiring multiplication by units. As far as the author knows, it is possible that such a result holds, or even such a result where the map shrinks in some directions, but where total volume expands. However, these results are harder to prove; we will do our best to indicate why in the course of the proof. 

\begin{proof} Suppose that $A \subseteq D_K$ is a set of positive measure satisfying \eqref{eq:Tinvariance}. We want to show that $A$ must have measure $\lambda(D_K)$.

As a subset of $\R^s \times \C^t$, we can treat $A$ as a subset of $\R^n$ (where \( n =s+2t \) is the degree of the number field), and hence we can apply the Lebesgue density theorem to say that $A$ must have a density point $\z$. That is, for any $\delta > 0$ we can find $E > 0$ such that for all balls $B(\z,\varepsilon)$ of radius $\varepsilon < E$, we have that
$$\frac{\lambda(A^C \cap B(\z,\varepsilon))}{\lambda(B(\z,\varepsilon))} < \delta.$$

For each $\delta$, consider $\varepsilon = e^{-b} < E$, where $b \in \N$, and take the set $B(\z,\varepsilon)$. For a map $T_{\alpha,\beta}$, we define a map $\tilde{T}_{\alpha,\beta} : \R^n \to \R^n$ which is just the map $T_{\alpha,\beta}$ without reducing mod $\iota(\OK_K)$. We now claim there exist $i, i_1, \ldots, i_r \in \Z_{\geq 0}$ such that if we define
$$T := \tilde{T}_{u_r,0}^{i_r} \circ \cdots \circ \tilde{T}_{u_1,0}^{i_1} \circ \tilde{T}_{\pi,\kappa}^i,$$
then the set $T(B(\z,\varepsilon))$ is such that
$$T(B(\z,\varepsilon)) + \iota(\gamma) \supseteq D_K$$
for some $\gamma \in \OK_K$ and such that we have
$$\lambda(T(B(\z,\varepsilon))) < C_{K,\pi}$$
for some constant $C_{K,\pi}$ depending on $K$ and $\pi$, but \emph{not} on $\varepsilon$. (This second property is the part which relies on the multiplication by units, and is important.)

Our ball $B(\z,\varepsilon)$ has volume $C_1 \varepsilon^n$, where $C_1$ depends only on $K$. It also contains a ``box'' $\Ball(\z,\varepsilon)$ given by
$$\Ball(\z,\varepsilon) = \prod_{\rho \in \Sigma} B \left( z_{\rho}, \tfrac{\varepsilon}{\sqrt{r+1}} \right)$$

with volume $C_2 \varepsilon^n$, where $C_2$ also depends only on $K$.

If we apply $T$ to the box $\Ball(\z,\varepsilon)$, we find that
$$T(\Ball(\z,\varepsilon)) = \prod_{\rho \in \Sigma} B_{\rho} \left(w_{\rho}, \frac{\varepsilon}{\sqrt{r+1}} \bigg| \rho \Big( u_1^{i_1} \cdots u_r^{i_r} \pi^i \Big) \bigg|^{\theta_{\rho}} \right)$$
for some $\mathbf{w}$, where $\theta_{\rho} = 1$ for real $\rho$, and $\frac{1}{2}$ for complex $\rho$.

Let $\{ L_{\rho} \}_{\rho \in \Sigma}$ be elements of $\R_{>0}$ such that 
$$D_K \subset \prod_{\rho \in \Sigma} B_{\rho}(0, L_{\rho}).$$

Then to guarantee that $T(B(\z,\varepsilon)) \supset D_K$, we can just ensure that
$$\frac{\varepsilon}{\sqrt{r+1}} \Big| \rho \big( u_1^{i_1} \cdots u_r^{i_r} \pi^i \big) \Big|_{\rho}^{\theta_{\rho}} > L_\rho$$
for each $\rho \in \Sigma$. We also want $i$ to be as small as possible, as the factors of $\pi$ in our map $T$ are the only factors which change the volume (by a factor of $\Nm(\pi)$).

Explicitly indexing our $\rho$, taking logarithms and rearranging gives
$$\sum_{k=1}^r i_k \log|\rho_j(u_k)| + i \log|\rho_j(\pi)| > \theta_j^{-1}\left(\log(L_j) + \frac{\log(r+1)}{2} + \log(\varepsilon^{-1})\right).$$

Writing
$$\alpha_{jk} = \log|\rho_j(u_k)|, \quad \Lambda_j = \log|\rho_j(\pi)|, \quad \ell_j = \theta_j^{-1}\left(\log(L_j) + \frac{\log(r+1)}{2}\right),$$
and noting that $\log \varepsilon^{-1} = \log e^b = b$, we can write these as a matrix equation (where the inequalities are just considered row-wise):
$$\left( \begin{array}{cccc} \alpha_{1,1} & \cdots & \alpha_{1,r} & \Lambda_1 \\ \vdots & \ddots & \vdots & \vdots \\ \alpha_{r+1,1} & \cdots & \alpha_{r+1,r} & \Lambda_{r+1} \end{array} \right) \left( \begin{array}{c} i_1 \\ \vdots \\ i_r \\ i \end{array} \right) > \left( \begin{array}{c} \ell_1 + \theta_1^{-1} b \\ \vdots \\ \ell_{r+1} + \theta_{r+1}^{-1} b \end{array} \right).$$

First, we consider $i_k, i \in \R$, change the inequality to an equality and solve. We can do this easily enough: first, we employ row reduction, adding a copy of each of the first \(r\) rows to row \(r+1\). Noting that
\[
\sum_{j=1}^{r+1} \log|\rho_j(x)| = \log \left| \prod_{j=1}^{r+1} \rho_j(x) \right| = \log |N(x)|
\]
and hence that
\[
\sum_{j=1}^{r+1} \alpha_{jk} = \log|N(u_k)| = 0 \quad \hbox{and} \quad \sum_{j=1}^{r+1} \Lambda_j = \log|N(\pi)|,
\]
we get
\[
\left( \begin{array}{cccc} \alpha_{1,1} & \cdots & \alpha_{1,r} & \Lambda_1 \\ \vdots & \ddots & \vdots & \vdots \\ \alpha_{r,1} & \cdots & \alpha_{r,r} & \Lambda_r \\ 0 & \cdots & 0 & \log|N(\pi)| \end{array} \right) \left( \begin{array}{c} i_1 \\ \vdots \\ i_r \\ i \end{array} \right) = \left( \begin{array}{c} \ell_1 + \theta_1^{-1} b \\ \vdots \\ \ell_r + \theta_r^{-1} b \\ \sum_{j=1}^{r+1} (\ell_{j} + \theta_{j}^{-1} b) \end{array} \right).\]

Now this equation clearly has solutions: expanding along the bottom row, the determinant of the matrix is seen to be non-zero, being
\[
\log|N(\pi)| \left| \begin{array}{ccc} \alpha_{1,1} & \cdots & \alpha_{1,r} \\ \vdots & \ddots & \vdots \\ \alpha_{r,1} & \cdots & \alpha_{r,r} \end{array} \right|,
\]
where the second quantity is just the \emph{regulator} of the number field, which is known to be positive. Furthermore, we see that the solution must have
\[
\log|N(\pi)|i = \sum_{j=1}^{r+1} (\ell_j + \theta_j^{-1} b),
\]
and hence
\[
i = \frac{\ell_1 + \cdots + \ell_{r+1}}{\log|\Nm(\pi)|} + \frac{bn}{\log|\Nm(\pi)|}.\]

Taking the floor of each of the components of this solution vector gives us integers. But then we only need a finite number of steps $S$ (independent of $\varepsilon$) in the $\pi$-direction to get inside our required region.

So for some $S \in \N$ not depending on $\varepsilon$, we always have a solution with
$$i \leq \frac{\ell_1 + \cdots + \ell_{r+1}}{\log|\Nm(\pi)|} + \frac{bn}{\log|\Nm(\pi)|} + S.$$

Now we want to see whether applying $T$ keeps the measure of $T(B(\z,\varepsilon))$ below some constant $C_{K,\pi}$. The measure of $T(B(\z,\varepsilon))$ is given by
\begin{align*} \vol(T(B(\z,\varepsilon))) &= |\Nm(\pi)|^i  C_1 \varepsilon^n \\ &\leq |\Nm(\pi)|^{S}  |\Nm(\pi)|^{\frac{\ell_1 + s + \ell_{r+1}}{\log|\Nm(\pi)|}}  |\Nm(\pi)|^{\frac{bn}{\log|\Nm(\pi)|}}  C_1 \varepsilon^n \\ &= C_1 L_1 s L_s L_{s+1}^2 s L_{s+t}^2 (r+1)^{\frac{r+1}{2s}} |\Nm(\pi)|^{S},\end{align*}
which depends only on $K$ and $\pi$ as required.

Now, our map $T$ just expands the measure of a set by a factor of $\Nm(\pi)^i$, and hence we have
$$ \frac{\lambda(T(A^C \cap B(\z,\varepsilon)))} {\lambda(T(B(\z,\varepsilon)))} = \frac{ \Nm(\pi)^i \lambda(A^C \cap B(\z,\varepsilon)) }{ \Nm(\pi)^i \lambda(B(\z,\varepsilon)) } < \delta$$
and therefore
$$ \lambda(T(A^C \cap B(\z,\varepsilon))) < \delta \lambda(T (B(\z,\varepsilon)) ). $$

So the sets 
$$T(A \cap B(\z,\varepsilon)) = T(A) \cap T(B(\z,\varepsilon)) \aand T(B(\z,\varepsilon))$$
differ by a set of measure at most
$$\delta T(B(\z,\varepsilon)) \leq C_{K,\pi} \delta.$$

(Note that since we have a bound on the size of \(T(B(\z,\varepsilon))\) which is independent of \(\varepsilon\), we can bound the discrepancy by a scalar multiple of \(\delta\).)

If we now reduce mod $\iota(\OK_K)$, the measure of the difference between the resulting sets cannot increase, and hence the sets
$$T(A) \cap T(B(\z,\varepsilon)) \mod \iota(\OK_K)$$
and
$$T(B(\z,\varepsilon)) \mod \iota(\OK_K)$$
also differ by a set of measure at most $C_{K,\pi} \delta$. Then noting that $T(B(\z,\varepsilon)) \mod \iota(\OK_K)$ is just $D_K$ (since $T(B(\z,\varepsilon)) \supset D_K$) and that
$$T(A) \cap T(B(\z,\varepsilon)) \subseteq T(A) \subseteq A$$
by our assumption, we have that the difference between $A$ and $D_K$ has measure at most $C_{K,\pi} \delta$. Taking $\delta \to 0$ completes the proof. \end{proof}

Now we have all of the necessary lemmas to prove our theorem.

\begin{proof}[Proof of Theorem \ref{thm:zerooneK}] The proof of this theorem closely follows the proof of Theorem 1 in \cite{Gallagher1961}. As in Gallagher's paper, the main difficulty to be overcome is the restriction of \emph{coprimality} (in this paper, that restriction is baked in via our definition of the denominator function \( \dnm(\gamma) \)), and this is overcome by considering a decomposition
\[
 \App = \A(\pid) \cup \B(\pid) \cup \C(\pid)
\]
for each of an infinite array of prime ideals \(\pid\). We start by showing that the sets \(\A(\pid)\) and \(\B(\pid)\) are always of measure \(0\) or \(\lambda(D_K)\), via applications of Lemma~\ref{lem:metricallytransitive}. Once this has been established, we tackle the case where all of the \(\A(\pid)\) and \(\B(\pid)\) are measure \(0\); then we can use periodicity of the sets \(\C(\pid)\) (which now all have exactly the same measure as \(\App\)) along with a limiting process to show that even in this case, we have that \(\App\) has measure \(0\) or \(\lambda(D_K)\).

Note first that any number field $K$ has infinitely many principal prime ideals. Let $\Omega_K$ be as in the statement of Lemma \ref{lem:metricallytransitive}. Then by Lemma \ref{lem:normbound}, there exists a constant $C$ such that for all principal prime ideals $\pid$ with $\Nm(\pid) > C$, we can find a generator $\pi$ of $\pid$ such that $|\rho(\pi)| > \Omega_K$ for all $\rho \in \Sigma$. From now on in this proof, we only work with such ideals, and the statements about ``all $\pid$'', etc., are taken to refer to all ideals satisfying these conditions. (Note that since the number of ideals of norm \(\leq C\) is finite, we exclude only finitely many of our prime ideals.)

Now, for each ideal $\pid = (\pi)$ and each $\nu \in \N$, we consider the approximation
\begin{equation}\label{eq:papprox} |z_{\rho} - \rho(\gamma)| < |\rho(\pi)|^{\nu - 1} \psi_{\rho}(\dnm(\gamma)) \hbox{ for all } \rho \in \Sigma.  \end{equation}

Define sets $\A(\pid^{\nu})$ by
\[
\A(\pid^{\nu}) = \left\{ \z \in \R^s \times \C^t \ \left| \ \begin{array}{c} \hbox{\(\z\) satisfies \eqref{eq:papprox} for infinitely} \\ \hbox{many \(\gamma\) with \(\pid \nmid \dnm(\gamma)\)} \end{array} \right. \right\},
\]
and define
$$\A^*(\pid) = \bigcup_{\nu \in \N} \A(\pid^{\nu}).$$

By Lemma \ref{lem:setsinboxes}, the set $\A(\pid^{\nu})$ has the same measure as $\A(\pid)$ for any $\nu \in \N$, and then by combining this with the fact that $\A(\pid^{\nu}) \subseteq \A(\pid^{\nu + 1})$ for any $\nu \in \N$, we find that $\A(\pid)$ has the same measure as the union $\A^*(\pid)$. 

We can now see that the map $T_{\pi,0}$ (as defined in the statement of Lemma~\ref{lem:metricallytransitive}) sends $\A(\pid^{\nu})$ into $\A(\pid^{\nu + 1})$, since we have
\[
|\rho(\pi)z_{\rho} - \rho(\pi \gamma)| = |\rho(\pi)||z_{\rho} - \rho(\gamma)| < |\rho(\pi)|^{\nu} \psi_{\rho}(\dnm(\gamma)),
\]
and if \( (\pi) \nmid \dnm(\gamma)\), then \(\dnm(\pi \gamma) = \dnm(\gamma) \). The same holds for the map $T_{u,0}$ for each $u \in \mathcal{F}_K$, since we have
\[
|\rho(u)| < \Omega_K < |\rho(\pi)|
\]
for all \(\rho \in \Sigma\) and \(u \in \mathcal{F}_K\). Hence all of these maps send $\A^*(\pid)$ into itself, and therefore (by Lemma \ref{lem:metricallytransitive}) the set $\A^*(\pid)$ must have measure $0$ or $\lambda(D_K)$. 

Now define $\B(\pid^{\nu})$ by
\[
\B(\pid^{\nu}) = \left\{ \z \in \R^s \times \C^t \ \left| \ \begin{array}{c} \hbox{\(\z\) satisfies \eqref{eq:papprox} for infinitely} \\ \hbox{many \(\gamma\) with \(\pid \mid\mid \dnm(\gamma)\)} \end{array} \right. \right\},
\]
where \( \pid \mid \mid \n \) means that \( \pid \mid \n \) but \( \pid^2 \nmid \n \), and let
$$\B^*(\pid) = \bigcup_{\nu \in \N} \B(\pid^{\nu}).$$

We can now see that the maps $T_{\pi,1}$ sends \(\B(\pid^{\nu})\) into \(\B(\pid^{\nu + 1})\), since 
\[
\left|\rho(\pi)z_{\rho} + \frac{1}{\rho(\pi)} - \rho\left(\pi \gamma + \frac{1}{\pi}\right)\right| = |\rho(\pi)||z_{\rho} - \rho(\gamma)| < |\rho(\pi)|^{\nu} \psi_{\rho}(\dnm(\gamma)),
\]
and if \( (\pi) \mid\mid \dnm(\gamma)\), then \(\dnm(\gamma) = \dnm(\pi \gamma + \frac{1}{\pi})\). The same also holds for $T_{u,0}$ for each $u \in \mathcal{F}_K$, and hence that $\B^*(\pid)$ must also have measure $0$ or $\lambda(D_K)$.

Finally, define sets $\C(\pid)$ by
\[\C(\pid^{\nu}) = \left\{ \z \in \R^s \times \C^t \ \left| \ \begin{array}{c} \hbox{\(\z\) satisfies \eqref{eq:papprox} for infinitely} \\ \hbox{many \(\gamma\) with \(\pid^2 \mid \dnm(\gamma)\)} \end{array} \right. \right\}.\]

Then we note that for any $\pid$, we have
$$\App = \A(\pid) \cup \B(\pid) \cup \C(\pid).$$

If any set $\A(\pid)$ or $\B(\pid)$ has non-zero measure, then it has measure $\lambda(D_K)$, and hence so does $\App$. So now assume that $\lambda(\A(\pid)) = \lambda(\B(\pid)) = 0$ for all $\pid$. Then we have
$$\lambda(\App) = \lambda(\C(\pid))$$
for all $\pid$. Next, note that if \(\z\) and \(\gamma\) satisfy
\[
|z_{\rho} - \rho(\gamma)| < \psi_{\rho}(\dnm(\gamma)) \hbox{ for all } \rho \in \Sigma
\]
with \(\pid^2 \mid \dnm(\gamma)\), then we have that
\[
\left|\left(z_{\rho} + \rho \left( \frac{\kappa}{\pi} \right) \right) - \rho \left( \gamma + \frac{\kappa}{\pi} \right) \right| = |z_{\rho} - \rho(\gamma)| < \psi_{\rho}(\dnm(\gamma)) \hbox{ for all } \rho \in \Sigma,
\]
and that \( \dnm(\gamma + \frac{\kappa}{\pi}) = \dnm(\gamma)\). So if \(\z \in \C(\pid)\), then we have $\z + \iota(\frac{\kappa}{\pi}) \in \C(\pid)$ for any $\kappa \in \OK_K$. 

%

Now, suppose that \(\App\) has positive measure (and hence \(\C(\pid)\) has positive measure for all \(\pid\)). Then we can consider a density point \(\z\) of \(\App\).

For each \(\pid\), we know that \(\iota(\pid^{-1})\) forms a lattice in \(\R^n\). Let \(D_{\pid}\) be a fundamental domain for this lattice which is contained entirely within \(D_K\) and whose interior contains \(\z\).

Since \(\C(\pid)\) is \(\frac{\OK_K}{\pi}\)-periodic, we have that 
\[
\lambda(C(\pid) \cap D_{\pid}) = \lambda(\C(\pid) \cap (D_{\pid} + \tfrac{\kappa}{\pi}))
\]
for any \(\kappa \in \OK_K\) (everything obviously being taken mod \(\iota(\OK_K)\)). So we have
\[
N(\pid) \lambda(\C(\pid) \cap D_{\pid}) = \bigcup_{\kappa \in \OK_K/\pid \OK_K} (\C(\pid) \cap (D_{\pid} + \tfrac{\kappa}{\pid})) = \lambda(\C(\pid) \cap D_K) = \lambda(\C(\pid)),
\]
and hence, since 
\[
N(\pid) = \frac{\lambda(D_K)}{\lambda(D_{\pid})},
\]
we have
\[
\lambda(D_K)\frac{\lambda(\C(\pid) \cap D_{\pid})}{\lambda(D_{\pid})} = \lambda(\C(\pid)).\]

But then by the Lebesgue density theorem, the left-hand side tends to \(\lambda(D_K)\) as \(N(\pid) \to \infty\). So we have \( \lambda(\C(\pid)) \to \lambda(D_K)\), and hence (since \(\lambda(\App) = \lambda(\C(\pid))\) for all \(\pid\)) we have \(\lambda(\App) = \lambda(D_K)\) as required.

 \end{proof}

\section{Overlap estimates}
\label{sec:overlap}
For an integral ideal $\n \in I_K$, define a set $\mathcal{A}_{\n}(\PSI)$ by
$$\mathcal{A}_{\n}(\PSI) = \left\{ \x \in D_K \ \left| \ \begin{array}{c} \hbox{there is some $\gamma \in K$ with $\dnm \gamma = \n$} \\ \hbox{such that $\gamma$ is a $\PSI$-good approximation to $\x$} \end{array} \right. \right\}.$$

Note that we can write
\[
\App = \limsup_{\n \in I_K} \mathcal{A}_{\n}(\PSI),
\]
where the \( \n \) are ordered by increasing norm (and those of equal norm are ordered arbitrarily).

In this section, we want to prove the following lemma about these sets:

\begin{lemma}\label{lem:overlapK} There exists some constant $C_K > 0$ such that for any two integral ideals $\m \neq \n$, we have
$$\lambda(\mathcal{A}_{\m}(\PSI) \cap \mathcal{A}_{\n}(\PSI)) \leq C_K \Nm(\m) \Nm(\n) \Psi(\m) \Psi(\n).$$\end{lemma}

\begin{proof} Define boxes $\Ball(\gamma,\PSI(\n))$ by
$$\Ball(\gamma,\PSI(\n)) := \prod_{\rho \in \Sigma} B_{\rho}(\rho(\gamma), \psi_{\rho}(\n)).$$

Then we have
\begin{equation}\label{eq:Anboxes}\mathcal{A}_{\n}(\PSI) = \left( \bigcup_{\substack{\gamma \in K \\ \dnm(\gamma) = \n}} \Ball(\gamma, \PSI(\n)) \right) \cap D_K.\end{equation}

So we can bound the measure of the overlap between the two sets by counting the number of pairs of boxes which overlap, and then bounding the measure of the overlap between any two boxes.

For $\Ball(\beta, \PSI(\m))$ and $\Ball(\gamma, \PSI(\n))$ to overlap, we need
$$B_{\rho}(\rho(\beta), \psi_{\rho}(\m)) \aand B_{\rho}(\rho(\gamma), \psi_{\rho}(\n))$$
to overlap for each $\rho \in \Sigma$. This certainly happens if we have
$$|\rho(\beta) - \rho(\gamma)|_{\rho} \leq 2 \max \{ \psi_{\rho}(\m) , \psi_{\rho}(\n)\}$$
for each $\rho \in \Sigma$. If we write
$$\Delta_{\rho} := 2 \max \{ \psi_{\rho}(\m) , \psi_{\rho}(\n) \},$$
then we want 
$$|\rho(\beta - \gamma)| \leq \Delta_{\rho}$$
for each $\rho \in \Sigma$.

Set $\beta - \gamma = \theta$. If we write $\g = \gcd(\m,\n)$, then we have $\theta \in \frac{\g}{\m \n} \OK_K$, and we also have $\theta \neq 0$, since $\dnm \beta \neq \dnm \gamma$. So we want non-zero $\theta \in \frac{\g}{\m \n} \OK_K$ satisfying
$$|\rho(\theta)| \leq \Delta_{\rho}$$
for all $\rho \in \Sigma$. 

Following the notation of \cite{LangBook} (specifically that of Chapter V, Theorem 0), we define a \(M_K\)-divisor \(\mathfrak{c}(v)\) by
\[
\mathfrak{c}(v) = \left\{ \begin{array}{rl}\Delta_{\rho} & \quad v = \rho \in \Sigma, \\ |\frac{\g}{\m\n}|_{\pid} & \quad v = |\cdot|_{\pid}. \end{array}  \right. \]

Then the number of potential $\theta$ we can have is (in the notation of \cite{LangBook}) given by \(\lambda(\mathfrak{c}) - 1\), and so immediately applying Theorem 0 of Chapter V in \cite{LangBook}, we find that the number of potential \(\theta\) is bounded above by
$$C \left( \prod_{\rho \in \Sigma} \Delta_{\rho} \right) \frac{\Nm(\m) \Nm(\n)}{\Nm(\g)},$$
where $C$ is some constant depending only on $K$. 

So now we want a bound for each $\theta$ on the number of pairs $(\beta,\gamma)$ such that
$$\beta, \gamma \in D_K, \quad \dnm \beta = \m, \quad \dnm \gamma = \n, \quad \beta - \gamma = \theta.$$

We can write $\m = \g \sid$ and $\n = \g \tid$, where $\gcd(\sid,\tid) = 1$. Now, write $\Nm(\m) = m$, $\Nm(\n) = n$, etc., and for each ideal $\aid$, define $\tilde{\aid}$ to be such that $\aid \tilde{\aid} = (a)$. Then
$$\beta = \frac{b}{gs} \aand \gamma = \frac{c}{gt},$$
where $b \in \tilde{\g} \tilde{\sid} \OK_K$ and $c \in \tilde{\g} \tilde{\tid} \OK_K$, and hence
$$tb - sc = gst \theta =: \Theta \in \tilde{\g} \tilde{\sid} \tilde{\tid} \OK_K.$$

Suppose two pairs $(b,c)$ and $(b',c')$ satisfy this. Then we have
$$tb - sc = \Theta = tb' - sc',$$
and hence
$$t(b-b') = s(c-c').$$

By comparing the left- and right-hand sides of this equation, we find that both sides lie in the space $\tilde{\g} (s) (t) \OK_K$, and hence any two solutions to $\beta - \gamma = \theta$ must have $\beta - \beta' \in \frac{1}{\g} \OK_K,$ giving at most $\Nm(\g)$ solutions. So there can be at most 
$$C \left( \prod_{\rho \in \Sigma} \Delta_{\rho} \right) \Nm(\m) \Nm(\n)$$
overlaps, with the size of each overlap being at most
$$\prod_{\rho \ \hbox{\scriptsize real}} 2 \min \{ \psi_{\rho}(\m) , \psi_{\rho}(\n) \} \prod_{\rho \ \hbox{\scriptsize complex}} \pi \min \{ \psi_{\rho}(\m) , \psi_{\rho}(\n) \}^2.$$

Then the total size of the overlap is bounded above by the maximum number of overlaps multiplied by the maximum size of any given overlap. This simplifies to
$$C 2^s \pi^t \Psi(\m) \Psi(\n) \Nm(\m) \Nm(\n),$$
and hence we have our result.\end{proof}
\section{Proving Theorem \ref{thm:duffinschaefferK}}
\label{sec:mainproof}
Now that we have our zero-one law and our overlap estimates, we can proceed to prove our main theorem.

\begin{proof}[Proof of Theorem \ref{thm:duffinschaefferK}] As before, note that we have
$$\App = \limsup_{\n \in I_K} \mathcal{A}_{\n}(\PSI).$$

First, we want to determine the measure of $\mathcal{A}_{\n}(\PSI)$. Looking at \eqref{eq:Anboxes}, we see this is a union of boxes of the same measure, and that the measure of a single box is given by
$$\left( \prod_{\hbox{\scriptsize $\rho$ real}} 2 \psi_{\rho}(\n) \right) \left( \prod_{\hbox{\scriptsize $\rho$ complex}} \pi \psi_{\rho}(\n)^2 \right) = 2^s \pi^t \Psi(\n).$$

By the definition of \(\Phi(\n)\), there are \(\Phi(\n)\) points of denominator \(\n\) inside \(D_K\); then around each we have a box of volume \(2^s \pi^t \Psi(\n)\). For each of these points, if any of the surrounding box \(B\) spills out of \(D_K\), then by simply translating by elements of the lattice \(\iota(\OK_K)\), we find a corresponding finite union of boxes 
\[
\tilde{\mathcal{B}} = \bigcup_{i=1}^h \tilde{B}_i
\]
all with centres \emph{outside} \(D_K\) such that we have
\[
\lambda(B \cap D_K^c) = \lambda(\tilde{\mathcal{B}} \cap D_K)
\]
(where \(D_K^c\) denotes the complement of the set \(D_K\)).

Conversely, for any box with centre outside \(D_K\) which intersects \(D_K\), we can always find a box with centre inside \(D_K\) which spills out of \(D_K\), and finitely many other boxes with centres outside \(D_K\) intersecting \(D_K\), such the measures again correspond. (That is, the union of the finitely many intersections with \(D_K\) from the boxes outside will have measure equal to that of the intersection of the box inside \(D_K\) with \(D_K^c\).)

Finally, by our boundedness condition \eqref{eq:boundedness}, all boxes are disjoint. To see this, suppose we have an overlap between two boxes in \(\mathcal{A}_{\n}(\PSI)\). Then for some pair \(\gamma \neq \gamma'\), both with denominator \(\n\), we must have
\[
2 \psi_{\rho}(\n) > \rho(\gamma) - \rho(\gamma')
\]
for all \(\rho \in \Sigma\). Then we would have
\[
\prod_{\rho \in \Sigma} 2 \psi_{\rho}(\n) > \prod_{\rho \in \Sigma} \rho(\gamma - \gamma'),
\]
and hence for some non-zero \(\theta \in \frac{1}{\n} \OK_K\) we would have
\[
\Nm(\theta) < \prod_{\rho \in \Sigma} 2 \psi_{\rho}(\n) \leq \frac{1}{\Nm(\n)}
\]
(where the last bound comes from \eqref{eq:boundedness}).

But this is clearly a contradiction, since non-zero elements of \(\frac{1}{\n} \OK_K\) have norm at least \( \frac{1}{\Nm(\n)} \). So all our boxes are disjoint, and hence we have that
$$\lambda(\mathcal{A}_{\n}(\PSI)) = 2^s \pi^t \Phi(\n) \Psi(\n).$$

We can now apply a standard measure-theoretic lemma (see for example Lemma 2.3 in \cite{HarmanBook}) to see that
$$\lambda(\mathcal{A}(\PSI)) \geq \limsup_{R \in \N} \left( \sum_{\substack{\n \in I_K \\ \Nm(\n) \leq R}} \lambda(\mathcal{A}_{\n}(\PSI)) \right)^2 \left( \sum_{\substack{\m,\n \in I_K \\ \Nm(\m),\Nm(\n) \leq R}} \lambda(\mathcal{A}_{\m}(\PSI) \cap \mathcal{A}_{\n}(\PSI)) \right)^{-1}.$$

Using the overlap estimates from Lemma \ref{lem:overlapK}, we have that
\begin{align*} \frac{\left( \sum_{N(\n) \leq R} \lambda(\mathcal{A}_{\n}(\underline{\psi})) \right)^2}{\sum_{N(\m),N(\n) \leq R} \lambda(\mathcal{A}_{\m}(\underline{\psi}) \cap \mathcal{A}_{\n}(\underline{\psi}))} &\geq C_K \frac{\left( \sum_{N(\n) \leq R} \Phi(\n) \Psi(\n) \right)^2}{\sum_{N(\m),N(\n) \leq R} N(\m) N(\n) \Psi(\m) \Psi(\n)} \\ 
&= C_K \frac{\left( \sum_{N(\n) \leq R} \Phi(\n) \Psi(\n) \right)^2}{\left( \sum_{N(\n) \leq R} N(\n) \Psi(\n) \right)^2} \\
 &= C_K \left( \frac{ \sum_{N(\n) \leq R} \Phi(\n) \Psi(\n)}{ \sum_{N(\n) \leq R} N(\n) \Psi(\n)} \right)^2.  \end{align*}

Then since we assumed \eqref{eq:limsup}, we have that $\lambda(\mathcal{A}(\underline{\psi})) > 0$, and hence (by Theorem \ref{thm:zerooneK}) we have that $\lambda(\mathcal{A}(\underline{\psi})) = \lambda(D_K)$ as required. \end{proof}


\begin{thebibliography}{50}
\bibitem{Aistleitner2014}
Christoph Aistleitner. A note on the Duffin--Schaeffer conjecture with slow divergence. \textit{Bulletin of the London Mathematical Society}, 46(1):164--168, 2014.

\bibitem{AistleitnerEtAl}
Christoph Aistleitner, Thomas Lachmann, Marc Munsch, Niclas Technau and Agamemnon Zafeiropoulos. The Duffin--Schaeffer conjecture with extra divergence. Preprint available at https://arxiv.org/abs/1803.05703, 2018.

\bibitem{ED2}
Victor Beresnevich, Glyn Harman, Alan Haynes and Sanju Velani. The Duffin--Schaeffer conjecture with extra divergence II. \textit{Mathematische Zeitschrift}, 275(1-2): 127--133, 2013.

\bibitem{Cantor1965}
D. G. Cantor.  On the elementary theory of diophantine approximation over the ring of adeles I. \textit{Illinois J. Math.}, 9:677--700, 1965.

\bibitem{DuffinSchaeffer1941}
R.J. Duffin and A.C. Schaeffer. Khintchine's problem in metric Diophantine approximation. \textit{Duke Mathematical Journal}, 8:243--255, 1941.

\bibitem{Erdos1970}
P. Erd\H{o}s. On the Distribution of the Convergents of Almost All Real Numbers. \textit{J. Number Theory}, 2:425--441, 1970.

\bibitem{Gallagher1961}
P.X. Gallagher. Approximation by reduced fractions. \textit{J. Math. Soc. Jpn}, 13:342--345, 1961.

\bibitem{HarmanBook}
Glyn Harman. \textit{Metric number theory}. Oxford University Press, 1998.

\bibitem{Harman1990}
G. Harman. Some cases of the Duffin and Schaeffer conjecture. \textit{Quart. J. Math. Oxford Ser. (2)}, 41(164):395--404, 1990.

\bibitem{ED1}
Alan K. Haynes, Andrew D. Pollington and Sanju L. Velani. The Duffin--Schaeffer Conjecture with extra divergence. \textit{Math. Ann.}, 353(2):259--273, 2012.

\bibitem{Khintchine1924}
A. Khintchine. Einige S\"{a}tze \"{u}ber Kettenbr\"{u}che, mit Anwendungen auf die Theorie der Diophantischen Approximationen. \textit{Math. Ann.}, 92:115--125, 1924.

\bibitem{KMDSProof}
D. Koukoulopoulos and J. Maynard. On the Duffin--Schaeffer conjecture. Preprint available at https://arxiv.org/abs/1907.04593, 2019.

\bibitem{LangBook} Serge Lang. \textit{Algebraic number theory}. Springer-Verlag New York, Inc., 1986.

\bibitem{NakadaWagner1991}
H. Nakada and G. Wagner. Duffin--Schaeffer theorem of diophantine approximation for complex numbers. \textit{Ast\'{e}risque}, 198--200:259--263, 1991.

\bibitem{Palmer2016}
Matthew Palmer. \textit{Diagonal Diophantine approximation and the Duffin--Schaeffer theorem in number fields}. PhD thesis, University of Bristol, 2016.

\bibitem{PollingtonVaughan}
A.D. Pollington and R.C. Vaughan. The $k$-dimensional Duffin and Schaeffer conjecture. \textit{Mathematika}, 37:190--200, 1990.

\bibitem{Vaaler1978}
Jeffrey D. Vaaler. On the metric theory of Diophantine approximation. \textit{Pacific Journal of Mathematics}, 76(2): 527--539, 1978.

\end{thebibliography}
\end{document}